\documentclass[a4paper,10pt,reqno]{amsart}
\pdfoutput=1
\usepackage{amsthm,amsmath,amssymb}
\usepackage{mathrsfs}	
\usepackage{amsrefs}	
\usepackage{amsfonts}
\usepackage[T1]{fontenc}
\usepackage[all]{xy}
\usepackage{enumitem}
\usepackage{caption}
\usepackage{graphicx}
\usepackage[breaklinks]{hyperref} 
\usepackage{comment}
\usepackage{pgf,tikz}
\usepackage{color}
\usetikzlibrary{arrows,decorations.pathmorphing}

\definecolor{linkcolor}{rgb}{0,0,0.675}%
\hypersetup{colorlinks=true,linkcolor=linkcolor,citecolor=linkcolor,urlcolor=linkcolor,bookmarksdepth=2}%
\newtheorem{theorem}{Theorem}[section]
\newtheorem{lemma}[theorem]{Lemma}
\newtheorem{corollary}[theorem]{Corollary}
\newtheorem{proposition}[theorem]{Proposition}
\theoremstyle{definition}

\newtheorem{example}[theorem]{Example}
\newtheorem{notation}[theorem]{Notations}
\setenumerate{topsep=5pt minus 3pt,itemsep=0pt,leftmargin=25pt,listparindent=\parindent}
\setenumerate[1]{label=\normalfont(\arabic{enumi})}
\setenumerate[2]{label=\normalfont(\alph{enumii})}
\setdescription{itemsep=0pt,leftmargin=25pt,listparindent=\parindent}
\renewcommand*{\P}{\mathbb P}
\newcommand*{\Hom}{\operatorname{Hom}}

\newcommand*{\Ext}{\operatorname{Ext}}

\newcommand*{\Cl}{\operatorname{Cl}}
\newcommand*{\Pic}{\operatorname{Pic}}

\newcommand*{\Db}{\operatorname{D}^{\sf b}}

\newcommand*{\Z}{\mathbb{Z}}

\newcommand*{\Q}{\mathbb{Q}}
\newcommand*{\C}{\mathbb{C}}

\newcommand*{\g}{{\sf g}} 
\renewcommand*{\v}{{\sf ver}} 
\let\op\operatorname
\title{On derived categories of nonminimal Enriques surfaces}
\author{Yonghwa Cho}
\date{}

\address{Department of Mathematical Sciences, KAIST, 291 Daehak-ro, Yuseong-gu, Daejeon 305-701, Korea}

\email{yonghwa.cho@kaist.ac.kr}

\begin{document}
%
%
%
%
	\begin{abstract}
		By Orlov's formula, the derived category of blow up must contain the original variety as a semiorthogonal component. This arises an interesting question: does there exist a variety $X$ such that $\operatorname{D}^{\sf b}(X)$ does not admit an exceptional collection of maximal length, but $\operatorname{D}^{\sf b}(\operatorname{Bl}_{x} X)$ admits such a collection? We give such an example where $X$ is a minimal Enriques surface.
	\end{abstract}
	\maketitle
	\section{Introduction}
		This short note is to give an answer to the question posed in \cite[Remark~3.12]{Vial:Exceptional_NeronSeveriLattice} on the existence of exceptional collections of maximal length in the derived categories of nonminimal Enriques surfaces. In his recent article\,\cite{Vial:Exceptional_NeronSeveriLattice}, Vial proves that an algebraic surface $S$ with $p_g=q=0$ admits numerically exceptional collections of maximal length if and only if one of the following is true:
		\begin{enumerate}
			\item $S$ is minimal and of Kodaira dimension $-\infty$ or $2$;
			\item $S$ is one of the Dolgachev surfaces of type $X_9(2,3)$, $X_9(2,4)$, $X_9(3,3)$, $X_9(2,2,2)$;
			\item $S$ is nonminimal.
		\end{enumerate}
		By this criterion, an (minimal) Enriques surface never has an exceptional collection of maximal length, but there still remains a possibility that its blow up has an exceptional collection of maximal length. Indeed, it turns out that there exist such examples:
		\begin{theorem}[see Theorem~\ref{thm: Main Thm}]\label{thm: Main Thm_Intro}
		There exist an Enriques surface $S'$ such that the blowing up at a general point gives a surface $S$ whose derived category admits a semiorthogonal decomposition of $13$ line bundles together with a triangulated category $\mathcal A$ satisfying $K_0(\mathcal A) = \Z/2\Z$.
		\end{theorem}
		By the formula due to Orlov\,\cite{Orlov:ProjBundle}, we have two very different-looking semiorthogonal decompositions
		\[
			\Db(S) = \langle \mathcal O_{E}(1),\ \Db(S') \rangle = \langle \mathcal A, E_1,\,\ldots,\,E_{13} \rangle,
		\]
		where $E$ is the exceptional divisor of $S \to S'$. It seems a very intriguing question to ask how these semiorthogonal components can be compared.
		
		In Section~\ref{sec: Prelim} we briefly explain notions related to the exceptional collections on algebraic surfaces. Section~\ref{sec: Construction} deals with the construction method of the nonminimal Enriques surfaces which appear in Theorem~\ref{thm: Main Thm_Intro} and the technical parts, including proofs, are discussed in Section~\ref{sec: Technical Proof}. The theoretical backgrounds on Sections~\ref{sec: Construction}--\ref{sec: Technical Proof} are developed in \cite{ChoLee:ExcColl_Dolgachev}, but these have been applied to simpler setup in this article. For this reason, we expect that this example is more comprehensible than the one in \cite{ChoLee:ExcColl_Dolgachev}.
		\begin{notation}\ 
			\begin{enumerate}[topsep=0pt]
				\item Everything is defined over the field of complex numbers, except $Y_\Q$ in the proof of Lemma~\ref{lem: Permutation lemma}.
				\item Let $\mu_r = \langle \zeta_r \rangle$ be the multiplicative group which is generated by the primitive $r$\textsuperscript{th} root of unity. The group action
				\[
					\mu_r \times \C^n \to \C^n,\quad \zeta_r \cdot (x_1,\ldots,x_n) = (\zeta_r^{a_1} x_1,\ldots, \zeta_r^{a_n} x_n)
				\]
				defines the quotient space $\C^n / \mu_n$, which we will denote by $\C^n / \frac{1}{r}(a_1,\ldots,a_n)$.
				\item For a scheme $T$ of finite type over $\C$ and a point $P \in T$, $(P \in T)$ denotes the analytic germ.
				\item Except stated otherwise, the equality between divisors indicates the linear equivalence relation. Also, we say $D$ is effective if $D$ is linearly equivalent to an effective divisor, or equivalently, $h^0(D) > 0$.
			\end{enumerate}
		\end{notation}

	\section{Preliminaries}\label{sec: Prelim}
		Let $X$ be a nonsingular projective variety over $\C$ and let $\Db(X)$ be the bounded derived category of coherent sheaves on $X$. An \emph{exceptional collection} is an ordered collection of objects $E_1,\ldots,E_k \in \Db(X)$ satisfying the following conditions:
		\[
			\Hom_{\Db(X)}(E_i , E_j[p]) \left\{
				\begin{array}{ll}
					\C & i=j,\ p=0 \\
					0 & i=j,\ p \neq 0 \\
					0 & i>j
				\end{array}
			\right.
		\]
		This notion is motivated from the decomposition problem in derived categories. In general, a triangulated category $\mathcal T$ admits a \emph{semiorthogonal decomposition} $\langle \mathcal T_1,\ldots,\mathcal T_k \rangle$ if
		\begin{enumerate}
			\item $\mathcal T_1,\ldots,\mathcal T_k$ are full triangulated subcategories of $\mathcal T$;
			\item the smallest full triangulated subcategory containing $\mathcal T_1,\ldots,\mathcal T_k$ is $\mathcal T$;
			\item $\Hom_{\mathcal T}( T_i,T_j ) = 0$ for each $i>j$ and $T_i\in \mathcal T_i$, $T_j \in \mathcal T_j$.
		\end{enumerate}
		If $\mathcal T = \Db(X)$ and $E_1,\ldots,E_k$ is an exceptional collection, then there exists a semiorthogonal decomposition
		\[
			\mathcal T = \langle \mathcal A, E_1,\ldots, E_k \rangle
		\]
		where $\mathcal A = \langle E_1,\ldots, E_k \rangle^\perp$ is the full triangulated subcategory generated by the objects
		\[
			\{ A \in \Db(X) : \Hom_{\Db(X)}(E_i,A[p]) = 0,\ i=1,\ldots,k,\ p \in \Z \}
		\]
		In practical situations, the $\Hom$-groups in the derived category has a geometric interpretation. Indeed, for any coherent sheaf $\mathcal F$ on $X$, we can regard $\mathcal F$ as an objects in $\Db(X)$ by considering the complex $(\ldots \to 0 \to \mathcal F \to 0 \to\ldots)$ concentrated at degree zero, then for coherent sheaves $\mathcal F_1$ and $\mathcal F_2$ one has
		\begin{equation}\label{eq: Hom in Derived categories}
			\Hom_{\Db(X)} (\mathcal F_1,\mathcal F_2[p]) \simeq \Ext^p_X(\mathcal F_1,\mathcal F_2).
		\end{equation}
		An exceptional object in $\Db(X)$ contributes a $\Z$-direct summand in the group $K_0(X)$. Thus, if $E_1,\ldots,E_k$ is an exceptional collection, then $K_0(X) = \Z^{\oplus k} \oplus K_0(\mathcal A)$ where $\mathcal A = \langle E_1,\ldots, E_k \rangle^\perp$. For this reason, the length of an exceptional collection is bounded by the rank of $K_0(X)$. If $X$ is an algebraic surface with $\op{CH}^2(X) \simeq \Z$, then it is known that $K_0(X) \simeq \Z^{\oplus 2} \oplus \Pic X$\,(see \cite[Lemma~2.7]{GalkinShinder:Beauville}). Hence, for $S$ as in Theorem~\ref{thm: Main Thm_Intro},
		\[
			K_0(S) \simeq \Z^{\oplus 2} \oplus \Pic S \simeq \Z^{\oplus 13} \oplus \Z/2\Z,
		\]
		thus the length of any exceptional collection does not exceed $13$. Also, once we establish such a collection, say $E_1,\ldots,E_{13}$, then the orthogonal category $\mathcal A:= \langle E_1,\ldots,E_{13} \rangle^\perp$ must satisfy $K_0(\mathcal A) = \Z/2\Z$, and in particular, $\mathcal A \not\simeq 0$.
	\section{Construction method}\label{sec: Construction}
		We begin with explaining the method to construct Enriques surfaces by $\Q$-Gorenstein smoothing. The method is originally developed in \cite{LeePark:SimplyConnected}. Also, the paper contains the construction of Enriques surfaces as an example\,(see \cite[Example~2]{LeePark:SimplyConnected}). Here, we give a minimal description to establish the notations for the future use. Let $h_1,\,h_2 \in S:= \C[x,y,z]$ be homogeneous cubics which define nodal cubics on the plane. Assume $h_1 \cap h_2$ are nine distinct points. Then the pencil $\mathfrak p := \lvert \lambda_1 h_1 + \lambda_2 h_2 \rvert$ defines $\varphi_\mathfrak p \colon \P^2 \dashrightarrow \P^1$ whose resolution of indeterminacy is the rational elliptic fibration $f' \colon Y' \to \P^1$. There are two special fibers which corresponds to $(h_i=0)$ for $i=1,2$. Let $Y \to Y'$ be the blow up at the nodal points of $(h_i=0)$. It ends up with the rational elliptic surface $f\colon Y \to \P^1$ with the following dual graph of divisors.
	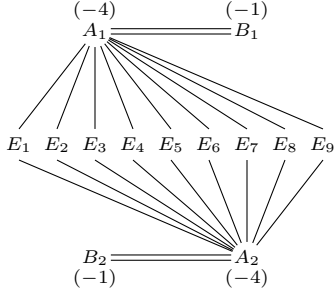
\begin{figure}[h!]
		\centering
		\begin{tikzpicture}
			\tikzset{shape=circle, inner sep=0pt}
			\draw(-1,1.5) node[anchor=center] (A1){$\scriptstyle A_1$};
			\draw(1,1.5) node[anchor=center] (A2){$\scriptstyle B_1$};
			\draw(1,-1.5) node[anchor=center] (B1){$\scriptstyle A_2$};
			\draw(-1,-1.5) node[anchor=center] (B2){$\scriptstyle B_2$};
			\draw[-] ([yshift=1pt] A1.east) -- ([yshift=1pt] A2.west);
			\draw[-] ([yshift=-1pt] A1.east) -- ([yshift=-1pt] A2.west);
			\draw[-] ([yshift=1pt] B1.west) -- ([yshift=1pt] B2.east);
			\draw[-] ([yshift=-1pt] B1.west) -- ([yshift=-1pt] B2.east);
			\node[yshift=2pt] at (A1.north) {$\scriptstyle (-4)$};
			\node[yshift=2pt] at (A2.north) {$\scriptstyle (-1)$};
			\node[yshift=-2pt] at (B1.south) {$\scriptstyle (-4)$};
			\node[yshift=-2pt] at (B2.south) {$\scriptstyle (-1)$};
			\foreach \x in {1,...,9}{
				\draw(0.5*\x-2.5,0) node[anchor=center] (E\x){$\scriptstyle E_\x$};
				\draw[-]	(E\x.north) -- (A1);
				\draw[-]	(E\x.south) -- (B1);
			}%
%
		\end{tikzpicture}\vskip-0.8\baselineskip%
		\caption{\label{fig: Minimal Configuration}Divisors on $Y$ and their intersections}%
	\end{figure}
	\par Here, $A_i$ is the proper transform of $(h_i=0)$ along $Y \to \P^2$, and $B_i$ is the exceptional divisor obtained by the blowing up $Y \to Y'$. Also, $E_1,\ldots,E_9$ are the divisors which appear in the blowing up $Y' \to \P^2$. An edge between two nodes implies that corresponding divisors meets with intersection number $1$. For example, it can be read $(A_i.B_i)=2$ from the graph. From $Y$, we can produce Enriques surfaces as follows.
	\begin{enumerate}[fullwidth, listparindent=\parindent, label=\bf{}Step~\arabic{enumi}., ref=\arabic{enumi}]
		\item Contract $A_1$ and $A_2$ to gain a singular surface $X$ with two $\frac 14 (1,1)$ singularities.
		\item Consider a $\Q$-Gorenstein smoothing $\mathcal X / (0 \in \Delta)$ of $X$, \textit{i.e.} a proper flat morphism $\mathcal X \to (0 \in \Delta)$ such that $\mathcal X_0 \simeq X$ and $\mathcal X_t$ is smooth for general $t \in \Delta \setminus\{0\}$.
		\item For general $t \in \Delta \setminus\{0\}$, $S := \mathcal X_t$ is an Enriques surface.
	\end{enumerate}
	One possible way to perform a blow up $S$ is to blow up the surface $Y$ in advance, and proceeds to Steps 1--3 described above. Hence, we add:
	\begin{enumerate}[fullwidth, listparindent=\parindent, label=\bf{}Step~\arabic{enumi}., ref=\arabic{enumi}]\setcounter{enumi}{-1}
		\item Blow up a point in $Y \setminus( A_1\cup A_2 \cup B_1 \cup B_2 \cup E_1 \cup \ldots \cup E_9 )$.
	\end{enumerate}
	By a slight abuse of notations, we keep call the resulting surface $Y$ and the respective divisors $A_1,\,A_2,\,B_1,\,B_2,\,\allowbreak E_1,\,\ldots,\,E_9$. Also, let $E_0$ be the exceptional divisor produced in Step~0.
	\begin{notation}
		Let $\pi \colon Y \to X$ be the contraction of $A_1,\,A_2$, let $p \colon Y \to \P^2$ be the blow down morphism, and let $H \subset \P^2$ be a line. Then, $\Pic Y$ is the free abelian group generated by
		\[
			\bigl\{ p^*H,\, E_0,\, E_1,\, E_2,\, \ldots,\, E_9,\, B_1,\, B_2 \bigr\}.
		\]
		Also define divisors $Q := p^*(2H)$, $\ell_i := p^*H - E_i$ for $i=1,\ldots,9$, and
		\[
			D_i = \left\{
				\begin{array}{ll}
					0 & i=0 \\
					- \ell_i + E_0 + B_1 & 1 \leq i \leq 9 \\
					-B_1 + E_0 & i=10 \\
					-Q + 3E_0 + 2B_1 & i= 11 \\
					2D_{11} & i = 12\,.
				\end{array}
			\right.
		\]
		The push forward along $\pi$ defines the $\Q$-Cartier Weil divisors on $X$. 
	\end{notation}
	Our main claim in this article is the following.
	\begin{theorem}\label{thm: Main Thm}
		Assume $h_1,h_2$ are general. There exist divisors $D_i^\g \in \Pic S$\,($i=0,\ldots,12$), which correspond to $D_i$, such that
		\[
			\mathcal O_S(D_0^\g),\, \mathcal O_S(D_1^\g),\, \ldots,\, \mathcal O_S(D_{12}^\g)
		\]
		is an exceptional collection in $\Db(S)$.
	\end{theorem}
	\section{The proof}\label{sec: Technical Proof}
	Our aim is to find divisors on $S$ which are comparable to the divisors on $X$. One of the natural attempts is to find a line bundle $\mathcal L \in \Pic \mathcal X$ so that the line bundles $\mathcal L\big\vert_S$ and $\mathcal L\big\vert_X$ share some information through $\mathcal L$. Unfortunately, this cannot be done for some line bundles on $S$. 
	\begin{proposition}\label{prop: Hacking Relative Homology Seq}
		There exists a short exact sequence
		\[
			0 \to \Pic S \to \Cl X \to H_1(M_1,\Z) \oplus H_1(M_2,\Z),
		\]
		where $M_i$ is the Milnor fiber of the smoothing $(P_i \in \mathcal X) / (0 \in \Delta)$. In particular, the image of $\Pic S \to \Cl X$ is exactly the set of Weil divisors $D_X$ on $X$ such that
		\[
			( D. A_i) \equiv 0\ (\op{mod} 2)\quad i=1,2
		\]
		where $ D \in \Pic Y$ is a proper transform of $D_X$ along $\pi$.
	\end{proposition}
	\begin{proof}
		For the proof, we refer to the arguments in \cite[\textsection 3.1]{ChoLee:ExcColl_Dolgachev}.
	\end{proof}
	Indeed, we have divisors on $Y$ which satisfies $(D.A_i) \equiv 0\ (\op{mod}2)$ but $\pi_*D$ is not Cartier, thus it is impossible to find $\mathcal L \in \Pic \mathcal X$ such that $\mathcal L \big\vert_X = \mathcal O(\pi_*D)$. We need a workaround, which modifies the total space $\mathcal X$ so that it is able to assign a line bundle on modified family $\tilde{\mathcal X}$ to each line bundle on $S$. This birational modification trick is developed in \cite[\textsection 3]{Hacking:ExceptionalVectorBundle}. There are two singularities, say $P_1,\, P_2$ in $X$, and these are isomorphic to $( 0 \in \C_{u,v}^2 / \frac 14(1,1) )$. By the change of variables $x = u^2$, $y = v^2$, $z = uv$, we get
	\[
		(P_i \in X) \simeq ( 0 \in (xy = z^2)) \subset (0 \in \C_{x,y,z}^3 / \frac12 (1,1,1)).
	\]
	The versal $\Q$-Gorenstein deformation is given by
	\[
		\mathcal X^\v := (xy = z^2 + t) \subset \C^3_{x,y,z} / \frac 12 (1,1,1) \times \Delta_t^\v,
	\]
	where $\Delta_t^\v \subset \C$ is a small complex disk centered at the origin. Locally, the ambient space $\C^3_{x,y,z} / \frac12(1,1,1) \times \Delta_t^\v$ can be identified to the toric variety associated with the fan $\Sigma$ generated by the standard basis of $\Z^4$ inside the lattice $N = \Z^4 + \Z \cdot \frac 12 (1,1,1,2)$. Let $\tilde \Sigma$ be the fan obtained by adding the ray $\Z \cdot \frac 12(1,1,1,2)$ to $\Sigma$. The resulting toric variety $\tilde \C$ admits the birational morphism $\Phi' \colon \tilde \C \to \C_{x,y,z}^3 / \frac12(1,1,1) \times \Delta_t^\v$. Let $\tilde {\mathcal X}^\v$ be the proper transform of $\mathcal X^\v$ and let $\Phi^\v \colon \tilde{\mathcal X}^\v \to \mathcal X^\v$ be the birational morphism induced by $\Phi'$.
	\begin{proposition}[cf. {\cite[\textsection3]{Hacking:ExceptionalVectorBundle}}]
		Let $\Phi \colon \tilde{\mathcal X} \to \mathcal X$ be the pullback of $\Phi^\v$ along $\mathcal X / (0 \in \Delta) \to \mathcal X^\v / (0 \in \Delta^\v)$. Then, $\tilde{\mathcal X}$ satisfies the following properties;
		\begin{enumerate}
			\item Over $P_i$, $W_i := \Phi^{-1}(P_i)$ is a projective plane and $\Phi$ is an isomorphism outside $\{P_1,\, P_2\}$.
			\item The central fiber over $0 \in \Delta$ is the union $Y \cup W_1 \cup W_2$, where $Y$ is the rational elliptic surface introduced above, and the scheme-theoretic intersection $W_i \cap Y$ is realized as $A_i$ in $Y$, while it is a conic in $W_i$.
		\end{enumerate}
	\end{proposition}
	Now, suppose that a divisor $D \in \Pic Y$ satisfies the conditions in Proposition~\ref{prop: Hacking Relative Homology Seq}. Let $2d_i := (D . A_i)$, then since $A_i$ is a conic curve in $W_i$, $\mathcal O_Y(D) \big\vert_{A_i} \simeq \mathcal O_{W_i}(d_i)\big\vert_{A_i}$. Let $\mathcal D_0$ be the glueing of $\mathcal O_Y(D)$, $\mathcal O_{W_1}(d_1)$, $\mathcal O_{W_2}(d_2)$, \textit{i.e.} the kernel of
	\[
		\mathcal O_Y(D) \oplus \mathcal O_{W_1}(d_1) \oplus \mathcal O_{W_2}(d_2) \to \mathcal O_{A_1}(2d_1) \oplus \mathcal O_{A_2}(2d_2),\qquad (s,s_1,s_2) \mapsto (s - s_1, s-s_2).
	\]
	Then, it can be easily proved that $\mathcal D_0$ is an exceptional line bundle on the reducible surface $\tilde{\mathcal X}_0 = Y \cup W_1 \cup W_2$. It is well-known that an exceptional vector bundle extends to a small neighborhood of deformation, thus shrinking $\Delta$, we can say that there exists a line bundle $\tilde{\mathcal D}$ such that $\tilde{\mathcal D}\big\vert_{\tilde{\mathcal X}_0} = \mathcal D_0$. Now, $\tilde{\mathcal D} \big\vert_{S}$ is a line bundle on $S$.
	\begin{notation}
		For a divisor $D \in \Pic Y$ as above, the divisor associated with the line bundle $\tilde{\mathcal D} \big\vert_{S}$ is denoted by $D^\g$.
	\end{notation}
	Since $\tilde{\mathcal D}$ is a flat family of line bundles, we have $\chi(D^\g) = \chi(\tilde{\mathcal D}_0)$. The latter can be computed via the short exact sequence
	\begin{equation}\label{eq: Short exact seq on Reducible fiber}
		0 \to \tilde {\mathcal D}_0 \to \mathcal O_Y(D) \oplus \mathcal O_{W_1}(d_1) \oplus \mathcal O_{W_2}(d_2) \to \mathcal O_{A_1}(2d_1) \oplus \mathcal O_{A_2}(2d_2) \to 0.
	\end{equation}
	By Riemann-Roch formula,
	\begin{equation}\label{eq: Euler Characteristic}
		\chi(\tilde{\mathcal D}_0) = \chi(D) + \frac{1}{2}d_1(d_1-1) + \frac{1}{2}d_2(d_2-1).
	\end{equation}
	\begin{lemma}
		For any divisor $D \in \Pic Y$ such that there exists an associated divisor $D^\g \in \Pic S$,
		\[
			(D^\g.K_S) = (D.E_0).
		\]
	\end{lemma}
	\begin{proof}
		Since $K_S$ is numerically equivalent to $E_0^\g$, it suffices to prove that $(D^\g. E_0^\g) = (D. E_0)$. By Riemann Roch formula, one can compute
		\[
			(D^\g.E_0^\g) = \chi(D^\g + E_0^\g) - \chi(D^\g).
		\]
		Let $\tilde {\mathcal D}_0$ be the line bundle obtained by glueing $\mathcal O_Y(D)$, $\mathcal O_{W_1}(d_1)$, $\mathcal O_{W_2}(d_2)$, and $\tilde {\mathcal E}_0$ be the line bundle obtained by glueing $\mathcal O_Y(E_0)$, $\mathcal O_{W_1}$, $\mathcal O_{W_2}$. Then, using the formula (\ref{eq: Euler Characteristic}) we get
		\begin{align*}
			(D^\g.E_0^\g) &= \chi(D^\g + E_0^\g) - \chi(D^\g) = \chi(\tilde{\mathcal D}_0 \otimes \tilde{\mathcal E}_0) - \chi(\tilde {\mathcal D}_0) \\
			&= \bigl( \chi( D + E_0 ) + \frac 12 d_1(d_1-1) + \frac12 d_2(d_2-1) \bigr) - \bigl( \chi(D) +  + \frac 12 d_1(d_1-1) + \frac12 d_2(d_2-1) \bigr) \\
			&= \chi(D+E_0) - \chi(D) = (D.E_0). \qedhere
		\end{align*}
	\end{proof}
	Combining all together, we get the intersection formula:
	\[
		(D^\g)^2 = (D.E_0) + 2 \chi( D) + d_1(d_1-1) + d_2(d_2-1) - 2.
	\]
	Note that all the information from the right hand side can be read in $Y$. Now, it is just a matter of computations to derived the following intersection table: let $1 \leq i \neq j \leq 9$
	\begin{equation}\label{eq: Intersection table of components}
		\begin{array}{c| c c c c c}
			& Q^\g & \ell_i^\g & \ell_j^\g & B_1^\g & E_0^\g \\ \hline
		Q^\g & 22 & 10 & 10 & 3 & 0 \\
		\ell_i^\g & 10 & 2 & 3 & 1 & 0 \\
		\ell_j^\g & 10 & 3 & 2 & 1 & 0 \\
		B_1^\g & 3 & 1 & 1 & 0 & 0 \\
		E_0^\g & 0 & 0 & 0 & 0 & -1
		\end{array}
	\end{equation}
	Note that $B_1^\g - B_2^\g$ is numerically trivial, so the above table also includes the intersections involving $B_2^\g$.
	\begin{proposition}\label{prop: Intersection theory}
		The intersection matrix of divisors $\{D_i^\g\}_{i=1,\ldots,11}$ is given by
		\[
			\bigl( ( D_i^\g . D_j^\g ) \bigr)_{1 \leq i,j \leq 11} = \left [
				\begin{array}{cccc}
					-1 & \cdots & 0 & 0 \\
					\vdots & \ddots & \vdots & \vdots \\
					0 & \cdots & -1 & 0 \\
					0 & \cdots & 0 & 1
				\end{array}
			\right]\raisebox{-2\baselineskip}[0pt][0pt]{.}
		\]
		Also, it holds that $K_S =_{\sf num} D_1^\g + \ldots + D_{10}^\g - 3D_{11}^\g$, where $=_{\sf num}$ indicates the numerical equivalence relation.
	\end{proposition}
	\begin{proof}
		The intersection table is obtained immediately from (\ref{eq: Intersection table of components}). Furthermore,
		\begin{align*}
			\textstyle\sum_{i=1}^{10} D_i^\g - 3D_{11}^\g &=\textstyle ( -\sum_{i=1}^9 \ell_i^\g + 8B_1^\g + 10E_0^\g) - 3( -Q^\g + 2B_1^\g + 3E_0^\g ) \\
			&= \textstyle 3Q^\g - \sum_{i=1}^9 \ell_i^\g + 2B_1^\g + E_0^\g,
		\end{align*}
		and $\sum_{i=1}^9 \ell_i = \bigl( p^*(6H) \bigr)^\g + \bigl( p^*(3H) - \sum_{i=1}^9 E_i )\bigr)^\g = 3 Q^\g + A_0^\g$ where $A_0^\g$ is the general elliptic fiber of $S$ induced by the elliptic fibration of $Y \to \P^1$. This leads to
		\[
			\textstyle\sum_{i=1}^{10} D_i^\g - 3D_{11}^\g = -A_0^\g + 2B_1^\g + E_0^\g.
		\]
		Since $A_0^\g = 2B_1^\g$, we have $\sum_{i=1}^{10} D_i^\g - 3D_{11}^\g = E_0^\g =_{\sf num} K_S$.
	\end{proof}
	\begin{corollary}
		For $0 \leq j < i \leq 12$, $\chi(-D_i^\g + D_j^\g)=0$.
	\end{corollary}
	\begin{proof}
		This is an immediate consequence of Proposition~\ref{prop: Intersection theory} and Riemann-Roch theorem. \qedhere
	\end{proof}
	To prove Theorem~\ref{thm: Main Thm}, we have to prove that $h^p(-D_i^\g + D_j^\g)=0$ for each $p$ and $0 \leq j<i \leq 12$. By the above corollary, it suffices to prove only for $p=0,2$. By Serre duality, $h^2(-D_i^\g + D_j^\g)=h^0(K_S + D_i^\g - D_j^\g)$, thus understanding $h^0$ of divisors is enough to prove Theorem~\ref{thm: Main Thm}. This can be done using the short exact sequence (\ref{eq: Short exact seq on Reducible fiber}); it is easy to see that $h^0(\mathcal O_{W_i}(d_i)) \to h^0(\mathcal O_{A_i}(2d_i))$ is always surjective, thus
	\begin{equation}\label{eq: Semicontinuity on h^0}
		h^0(D^\g) \leq h^0(\tilde{\mathcal D}_0) = h^0(D) + h^0(\mathcal O_{W_1}(d_1)) + h^0(\mathcal O_{W_2}(d_2)) - h^0(\mathcal O_{A_1}(2d_1)) - h^0(\mathcal O_{A_2}(2d_2)).
	\end{equation}
	Before proceed to the proof, we give one remark explaining why we need to use a computer-based approach.
	\begin{example}\label{eg: h^0 depends on configuration}
		The divisor $-D_9^\g + D_0^\g = (p^*H - E_9 - E_0 - B_1)^\g$ can be obtained by deforming $D := p^*H - E_9 - E_0 - B_1$. Since $(D.A_1) = 0$ and $(D.A_2)=2$,
		\begin{align*}
			h^0(D^\g) &\leq h^0(D) + h^0(\mathcal O_{W_1}) + h^0(\mathcal O_{W_2}(1)) - h^0(\mathcal O_{A_1}) - h^0(\mathcal O_{A_2}(2)) \\
			&= h^0(D).
		\end{align*}
		Because of its divisor form, $h^0(D)$ depends on the configuration of $(h_i=0)$. Indeed, $h^0(D)=1$ if the points $p(E_9)$, $p(E_0)$ and $p(B_1)$ are colinear and is zero otherwise. If these three points are colinear, we just have $h^0(D^\g) \leq 1$, thus we cannot see the desired vanishing. In this simple example, we can present the possible values of $h^0(D)$ together with exact criterion, but it is getting complicated for other pairs. For example, to conclude $h^0(-D_{12}^\g + D_{10}^\g) = 0$, we will see that it suffices to prove that $h^0(D)=0$ where
		\[
			\textstyle p^*(16H) - 4\sum_{i \leq9}E_i - 6E_0 - 6A_2 - 6B_2.
		\]
		This means that we have to show there is no plane curve of degree $16$ which passes through $p(E_1),\,\ldots,\,p(E_9)$ 4 times for each, $p(E_0)$ 6 times, $p(A_2)$ 6 times, and $p(B_2)$ 6 times.
	\end{example}
	\begin{lemma}\label{lem: Perturbation of h_i}
		Let $D$ be a divisor of the form
		\begin{equation}\label{eq: divisors interpreted as plane curves}
			p^*(dH) - (\text{positive sum of }E_0,\, E_1,\, \ldots, \,E_9,\, B_1,\, B_2 ).
		\end{equation}
		For given particular $h_1,h_2$, assume $h^0(D) = N$. Then, $h^0(D) \leq N$ for general $h_1,h_2$.
	\end{lemma}
	\begin{proof}
		As explained in Example~\ref{eg: h^0 depends on configuration}, counting $h^0(D)$ reduces down to count the dimension of the space of plane curves passing through the prescribed positions. Consider the homogeneous equation $\sum_\alpha c_\alpha \mathbf{x}^\alpha$ of degree $d$, where $\alpha = (\alpha_x,\alpha_y,\alpha_z)$ is a $3$-tuple with $\alpha_x + \alpha_y + \alpha_z = d$ and $\mathbf{x}^\alpha = x^{\alpha_x} y^{\alpha_y}z^{\alpha_z}$. Then the positional conditions given by $D$ will imposes linear conditions on $(c_\alpha)$, thus we get a system of linear equations, say $M \mathbf{c} = 0$. Now, $h^0(D)$ is the dimension of the space of solutions of this system. If we perturb $h_1,h_2$ then it perturbs $M$, but the rank of matrices is a lower-semicontinuous function, so $\dim \op{ker}M = h^0(D)$ is upper-semicontinuous with respect to $h_1,h_2$.
	\end{proof}
	Now it suffices to prove Theorem~\ref{thm: Main Thm} for the particular choice of $h_1,h_2$. However, since we use computer-based approach, we still need another obstruction to overcome. Imagine the situation that we are given a divisor $D$ of the form (\ref{eq: divisors interpreted as plane curves}), and $h_1,h_2$ are explicitly given. We have to tell the computer the conditions imposed by the plane curve $p_*D$ as an ideal sheaf of $\P^2$. The problem is that it is extremely hard\,(perhaps impossible) to find an example of $h_1,h_2$ such that the points corresponding to $E_1,\ldots,E_9,B_1,B_2$ are defined over a subfield of $\C$ which is solvable by radicals over $\Q$. Unless this is possible, we cannot define the explicit ideals in computers. This problem can be resolved by observing some symmetric nature between $E_1,\ldots,E_9$. In the end, we will see that it suffices to find cubics such that only $p(E_9)$, $p(B_1)$, $p(B_2)$ are defined over $\Q$.
	\begin{lemma}\label{lem: Permutation lemma}
		Assume $h_1,\,h_2$ define general nodal plane cubics. Let $D$ be a divisor in the rational elliptic surface $Y$, and assume that in the expression of $D$ in terms of the $\Z$-basis $\{p^*H,\,E_1,\,\ldots,\,E_9,\,A_2,\,B_2,\,B_3\}$, the coefficients of $E_1,\,\ldots,\,E_9$ are same. Then, $h^p(D+E_i) = h^p(D+E_j)$ for any $p$ and any $1 \leq i,j \leq 9$.
	\end{lemma}
	\begin{proof}
		The statement is a slight variation of \cite[Lemma~5.8]{ChoLee:ExcColl_Dolgachev}, so we do not give a precise proof here. The main idea is the following: first, pick $h_1,h_2$ so that
		\begin{enumerate}
			\item the point $p(E_0)$, the cubics $h_1,h_2 \in \Q[x,y,z]$, and the nodes of them are defined over $\Q$;
			\item the ideal $(h_1,h_2)$ is prime in $\Q[x,y,z]$;
			\item the points in the intersection is written as the form $[x_i,y_i,1] \in \P^2_\C$, where $x_1,\ldots,x_9$ are Galois conjugate to each other, $y_1,\ldots,y_9$ are Galois conjugate to each other, and the irreducible polynomials of $x_i$ and $y_j$ are not the same.
		\end{enumerate}
		Now, we take $\tau \in \op{Aut}(\C/\Q)$ such that $\tau(x_i) = x_j$. By condition (1), $Y$ is indeed defined over $\Q$, \textit{i.e.} there exists a variety $Y_\Q$ defined over $\Q$ such that $Y = Y_\Q \times_\Q \C$. Consider the automorphism $\tau_Y := {\sf Id}_{Y_\Q} \times \tau$ of $Y$. Then, $\tau_Y$ permutes $E_1,\ldots,E_9$\,(sending $E_i$ to $E_j$), and fixes $E_0$, $B_1$, $B_2$. Thus $\tau_Y$ fixes $D$. In particular,
		\[
			h^p(D+E_j) = h^p(\tau_Y^*(D + E_j)) = h^p( D+ E_i). \qedhere
		\]
	\end{proof}
	By Lemmas~\ref{lem: Perturbation of h_i} and \ref{lem: Permutation lemma}, the main part of the proof reduces to the following statement:
	\begin{proposition}\label{prop: Computable Subcollection}
		There exists $h_1,h_2 \in \C[x,y,z]$ such that
		\[
			\mathcal O_S(D_0^\g),\, \mathcal O_S(D_9^\g),\, \mathcal O_S(D_{10}^\g),\,\mathcal O_S(D_{11}^\g),\,\mathcal O_S(D_{12}^\g)
		\]
		is an exceptional collection in $\Db(S)$.
	\end{proposition}
	\begin{proof}
		We choose $h_1 = (y-z)^2 z - x^3 - x^2 z$ and $h_2 = x^3 - 2xy^2 + 2xyz + y^2z$, and pick $p(E_0) = [4,9,6]$. We denote the number $h^p(-D_i^\g+D_j^\g)$ by $h^p_{i,j}$. First of all, we observe that the divisor $B_1^\g$ is nef, hence any divisor $D^\g \in \Pic S$ with $(D^\g . B_1^\g) < 0$ must have vanishing $h^0$. Also, by Serre duality, $(D_i^\g - D_j^\g \mathbin. B_1^\g) > 0$ implies $h^2_{ij}= 0$. Using this criterion, we get the vanishing of the following numbers:
		\[
			h^2_{9,0},\, h^0_{10,9},\, h^2_{11,0},\, h^2_{11,9},\, h^2_{11,10},\, h^2_{12,0},\, h^2_{12,9},\, h^2_{12,10},\, h^2_{12,11}.
		\]
		In what follows, we prove $h^p_{ij}=0$ for $i,j \in \{ 0,9,10,11,12\}$ with $j<i$ by taking suitable divisor $D \in \Pic Y$ such that $D$ deforms to either $-D_i^\g + D_j^\g$ of $K_S + D_i^\g - D_j^\g$, and by using (\ref{eq: Semicontinuity on h^0}).
		\[
			\begin{array}{c | l }
				\text{result} & \text{choice of }D \\ \hline
				h_{9,0}^0=0 & p^*(H) - E_9 - E_0 - B_1 \\
				h^2_{10,9}=0 & p^*(H) - E_9 + E_0 - B_1 - B_2 \\
				h_{11,0}^0=0 & p^*(5H) - \sum_{i \leq 9} E_i - 3E_0 - 2B_1 - 2B_2 \\ 
				h_{11,9}^0=0 & p^*(4H) - \sum_{i \leq 8} E_i - 2E_0 - B_1 - 2B_2 \\ 
				h_{11,10}^0=0 & p^*(5H) - \sum_{i\leq 9} E_i- 2E_0 - 3B_1 - 2B_2 \\ 
				h_{12,0}^0=0 & p^*(16H) - 4\sum_{i \leq9}E_i - 6E_0 - 6B_1 - 6B_2 \\ 
				h_{12,9}^0=0 & p^*(12H) - 3\sum_{i\leq8} E_i - 2E_9 - 5E_0 - 6B_1 - 4B_2 \\ 
				h_{12,10}^0=0 & p^*(10H) - 3\sum_{i\leq9} E_i -5E_0 - 5B_1 - 6B_2 \\ 
				h_{12,11}^0=0 & \text{same $D$ as in $h_{11,0}^0$}
			\end{array}		
		\]
		For the first two in the table, it is easy to verify $h^0(D)=0$ if the triples $(\, p(E_9), p(E_0), p(B_1)\,) $ and $(\, p(E_9), p(B_1), p(B_2) \, )$ are not colinear. For the rest part of the table, we use computer to find $h^0(D)=0$. The Macaulay2 scripts can be found in \cite{ChoLee:Macaulay2Enriques}.
	\end{proof}
	\begin{proof}[Proof of Theorem~\ref{thm: Main Thm}]
		The only thing remains to prove is $h^p_{i,j} = 0$ for $1 \leq j < i < 9$. This can be easily shown since (\ref{eq: Semicontinuity on h^0}) reads
		\[
			h^0_{i,j} \leq h^0(- E_i + E_j) = 0, \qquad h^2_{i,j} \leq h^0( K_Y + E_1 - E_2 ) = 0.
		\]
		Lemmas~\ref{lem: Perturbation of h_i}, \ref{lem: Permutation lemma} and Proposition~\ref{prop: Computable Subcollection} imply that
		\[
			\mathcal O_S(D_0^\g),\, \mathcal O_S(D_1^\g),\, \ldots,\, \mathcal O_S(D_{12}^\g)
		\]
		is an exceptional collection after a slight perturbation of $h_1$ and $h_2$.
	\end{proof}
{ \small
	\noindent{\bfseries Acknowledgement.} The main part of this work has been done when the author was visiting Universit\"at Bayreuth. He would like to thank the members of Lehrstuhl Mathematik VIII -- Algebraische Geometrie for their hospitality and for providing conducive working atmosphere during his visit.
}
\bibliography{nonminEnriques}
\end{document}